\documentclass[joc]{ipart}

\usepackage{amsthm,amsmath,amssymb}
\usepackage{enumerate}
\usepackage{tikz}

\pubyear{0000}
\volume{0}
\issue{0}
\firstpage{1}
\lastpage{1}
\arxiv{http://arxiv.org/abs/1606.00730}

\startlocaldefs
\newtheorem{theorem}{Theorem}
\newtheorem{problem}[theorem]{Problem}

\newtheorem{definition}[theorem]{Definition}
\def\bd{\mathbf{d}}
\def\bD{\mathbf{D}}
\def\bs{\mathbf{s}}
\def\J{\mathcal{J}}
\def\DS{\textbf{DS}}
\def\BF{\texttt{\bfseries BASKET FILLING}}
\def\3P{\texttt{\bfseries 3-PARTITION}}
\def\XY{neighborhood degree sum\ }
\endlocaldefs

\begin{document}
\begin{frontmatter}
\title{Not all simple looking degree sequence problems are easy\protect\thanksref{T1}}
\runtitle{{\em NP}-complete degree sequence  problems}
\thankstext{T1}{Both authors were supported partly by National Research, Development and Innovation Office – NKFIH, under the grants K 116769 and SNN 116095.}

\begin{aug}
\author{P\'eter L. Erd\H os \thanksref{cor}\ead[label=e1]{erdos.peter@renyi.mta.hu}
\ead[label=u1,url]{www.renyi.mta.hu/$ \sim $elp/}}
\thankstext{cor}{Corresponding author}
\address{Alfr\'ed R\'enyi Institute of Mathematics, Hungarian Academy of
Sciences, \\
Budapest, Hungary
\printead{e1}\\
\printead{u1}}
\and
\author{Istv\'an Mikl\'os \ead[label=e2]{miklos.istvan@renyi.mta.hu}
\ead[label=u2,url]{www.renyi.mta.hu/$ \sim $miklosi/}}
\address{Alfr\'ed R\'enyi Institute of Mathematics, Hungarian Academy of
Sciences, \\
Budapest, Hungary
\printead{e2}\\
\printead{u2}}
\runauthor{P.L. Erd\H{o}s and I. Mikl\'os}
\end{aug}
\begin{abstract}
Degree sequence (\DS) problems are around for at least hundred twenty years, and with the advent of network science, more and more complicated, structured \DS\ problems were invented. Interestingly enough all those problems so far are computationally easy. It is clear, however, that we will find soon computationally hard \DS\  problems. In this paper we want to find such hard \DS\ problems with relatively simple definition.

For a vertex $v$ in the simple graph $G$ denote $d_i(v)$ the number of vertices at distance exactly $i$ from $v$. Then $d_1(v)$ is the usual degree of vertex $v.$ The vector $\bd^2(G)=( (d_1(v_1), d_2(v_1)), \ldots,$ $(d_1(v_n), d_2(v_n))$ is the {\bf second order degree sequence} of the graph $G$. In this note we show that the problem to decide whether a sequence of natural numbers $((i_1,j_1),\ldots (i_n,j_n))$ is a second order degree sequence of a simple undirected graph $G$ is strongly {\em NP}-complete. Then we will discuss some further {\em NP}-complete \DS\ problems.
\end{abstract}
\begin{keyword}[class=MSC]
\kwd[Primary ]{05C07}
\kwd{60K35}
\kwd[; secondary ]{68R05}
\end{keyword}
\begin{keyword}
{\small degree sequences of simple graphs; second order degree sequences; {\bf basket filling} problem; {\bf \XY}}
\end{keyword}
\received{\smonth{1} \sday{1}, \syear{0000}}

\end{frontmatter}

\section{Introduction}
A network emerged from a complex, real-life problem can be considered known if one can determine its fundamental parameters. One way to ascertain that the chosen parameter set fully determines the important properties of the network is the following: one can randomly generate ensembles of synthetic networks compatible with the parameter set, then  evaluate the similarities and differences among the original network and the generated ones. One important prerequisite for this procedure is to check the feasibility of the actual values of the given parameter set.

For graphs the most simple parameter is the degree sequence. There are myriad results and algorithms dealing with degree sequences: for example Petersen, 1892; Senior, 1950; Tutte, 1954; Gale, 1957; Ryser, 1957; Havel, 1957; Erd\H{o}s and Gallai, 1960; Hakimi, 1962; Fulkerson, 1964; Edmonds, 1965; Berge, 1981; just to name some. Interestingly enough all those problems provide computationally easy algorithms.

Already the first few network applications pointed out fast, that the degree sequences cannot differentiate efficiently among different type of real-life networks: Gene regulatory networks and social networks  with shared degree sequence tend to be scale-free, however, they have very different assortativity properties. In regulatory networks, the large degree vertices are typically connected to small degree vertices, while in social networks, the large degree vertices are connected with large degree vertices. To cope with this phenomena, new graph construction problems have been introduced, the {\em Joint Degree Matrix} (JDM) and the {\em Partition Adjacency Matrix} (PAM) problems. (See papers \cite{pinar,JDM, Mihail} and \cite{skeleton}.) These problems are easy in the same way.

Mahadevan {\em et al.} introduced the $dk$ series of graphs that subsumes all the basic degree-based characteristics of networks of increasing detail (see \cite{dk}). It is defined as a collection of distributions of $G$'s subgraphs of size $d = 0, 1, \ldots, N$, in which nodes are labelled by their degrees in $G$. Namely, the $1k$-distribution is the degree sequence, the $2k$-distribution is the joint degree distribution. Thus, this approach naturally extends the degree sequence and the JDM problems and it is natural to consider the problem to construct a graph with prescribed $1k$, $2k$, $3k$, etc. distribution, and also natural to ask the computational complexity of these problems.

The problem of {\em NP}-completeness for graph construction in general was first suggested to us by Z. Toroczkai (\cite{toro}), which was posed later in \cite{Orsini}. Toroczkai also conjectured that the $dk$ series problem becomes $NP$-complete fast. Intuitively its reason is simple: very fast we have too many constrains for the variables. It is an interesting question: how complicated must  a degree sequence be to be computationally hard.

In this paper we show that the degree sequence construction may get {\em NP}-complete with much less constraints, namely the number of the first and (exactly) second neighbors of every vertex already constitutes such computationally hard problem.

Our proposed problem is also related to a problem on privacy issues of recommendation systems, see \cite{dorka}. Assume that a bipartite graph is given with its adjacency matrix $M$. The rows are the costumers, the columns are the items. While $M$ is unknown for privacy reasons, we know the expressions $M M^T$ and $M^T\! M$. In other words, for any vertex pair in the same vertex class (both for costumers and items) the number of their common neighbors is known (including for any vertex with itself, so the degree of this vertex is also given). The question is whether $M$ can be reconstructed from this data. The number of vertices $w \ne v$ for which  the number of common neighbors is not zero gives the second order degree of $v$. So we have all the data given in the second order degree sequence problem. However here we also have some extra data beyond the first and second order degrees. The complexity of deciding the graphicality of such matrix data is unknown to date.

\section{A brief survey on the complexity of \DS\ problems}
Below we survey briefly what is known about the complexity of degree sequence problems. Let $G$ be a simple, undirected graph, and let $\bd(G)$ be its {\bf degree sequence} denoted as $(d(v_1),\ldots, d(v_n)).$ It is well-known that it can be decided in polynomial time whether a given sequence $\bd'$ is {\bf graphical} with the greedy algorithm of Havel and Hakimi (see \cite{havel, hakimi}). Their algorithm can be easily extended to directed degree sequences and bipartite degree sequences. The Joint Degree Sequence problem can still be solved with a greedy algorithm in polynomial time, see \cite{Mihail}, \cite{pinar}, \cite{JDM} and \cite{Bassler}. It is interesting to mention that Tutte's $f$-factor theorem (\cite{T52, T54}) can be used directly to solve the degree sequence problem (see \cite{JSV}), but not for the Joint Degree Sequence problem.

When more constraints are introduced, there are not known greedy algorithms  to solve those degree sequence problems. However, Tutte's $f$-factor theorem and Edmonds' famous blossom algorithm (\cite{E1965a}) can be applied to solve such degree sequence problems efficiently. For example they can be used to find tripartite realizations of degree sequences with fixed vertex partitions, while no "direct" solution is known for these problems. Another example is graph realizations with a given number of edges crossing a given bipartition (see Erd\H{o}s {\it et. al.} \cite{skeleton}). Here the vertex set $V$ is equipped with a degree sequence $\bd(V)$ and a bipartition $V=U\biguplus W$ is given together with a  natural number $k$. We are looking for a graphical realization of $\bd(V)$ where the number of {\em crossing} edges between $U$ and $W$ is exactly $k.$ There is not known greedy type algorithm to solve this problem, but Edmonds' algorithm and some further considerations provide a polynomial time solution for it.

Generally speaking, the analogous hypergraph degree sequence problems are much harder. For example, there is no known good necessary and sufficient condition for the graphicality of a hypergraph degree sequence. The common thinking says that the majority of these problems should be {\em NP}-hard. Chv\'atal has already found   a {\em NP}-complete problem similar to hypergraph degree sequence questions in 1980 (see \cite{chvatal}): the {\em intersection pattern} of a hypergraph with $N$ hyper-edges is $N^2$ numbers, which give for all edge pairs the cardinality their intersection. He proved that the obviously defined intersection pattern problem is {\em NP}-complete even for 3-uniform hypergraphs.

In case of 3-uniform {\em linear} hypergraphs (no two edges have two points in common) the same problem becomes polynomially solvable (see Jacobson {\it et. al.} \cite{lehel1}). However, for general 3-uniform hypergraphs, if we are given the edge pairs with two points intersection, then  the corresponding decision problem becomes {\em NP}-complete again (see Jacobson {\it et. al.} \cite{lehel2}).

Other known {\em NP}-complete hypergraph degree sequence problems are due to Colburn, Kocay and Stinson (\cite{colburn}): for a $k$-uniform hypergraph $H=(V(H),\mathcal{E}(H))$ and  vertex $u\in V(H)$ consider the $(k-1)$-uniform hypergraph $H_u$ consists of $\{ F\setminus u : u\in F \in \mathcal{E}(H)\}$. We say that $H$ {\bf subsumes} all the $(k-1)$-uniform hypergraph $H_u$ for each $u \in V(H).$ The paper proved that the following two problems are {\em NP}-complete:
\begin{enumerate}[{\rm (1)}]
\item Given $n$ graphs (i.e. 2-hypergraphs) $g_1,\ldots, g_n,$ is there a 3-uniform hypergraph $H$ such that the subsumed graphs $H_i$ are $g_i$?
\item Given the degree sequences of $n$ graphs $g_1,\ldots, g_n,$ is there a 3-uniform hypergraph $H$ whose subsumed graphs $H_i$ have the same degree sequences?
\end{enumerate}

\medskip\noindent
The main purpose of this note is to find "simple" looking but {\em NP}-complete degree sequence type problems. For this end, we are looking modest restriction sets for the classical degree sequence condition.

One particular restriction set was introduced at latest in \cite{alavi} (Alavi {\it et.al.}): let $v$ be a vertex in the graph $G$. The value $d_k(v)$ is defined as the number of vertices at distance $k$ from $v$. Then $d_1(v)$ denotes the usual degree of vertex $v.$ Denote $\bd^2(G)=( (d_1(v_1), d_2(v_1)),\ldots,$ $(d_1(v_n),$ $d_2(v_n))$  or $=(\bd_1(G),\bd_2(G))$ the {\bf second order degree sequence} of the graph $G$.
\begin{definition}\label{def:SOD}
The {\bf second order degree sequence} problem is to decide whether a sequence of pairs of natural numbers $((i_1,j_1),\ldots (i_n,j_n))$ is a second order degree sequence of a simple undirected graph $G$.
\end{definition}

Recently this problem was revived. For example, Araujo-Pardo and her colleagues studied the possible relations among the sizes of $d_1(v)$ and $d_2(v)$ (\cite{pardo}).

Naturally one can study similar problems for second order degree sequences as for the more conservative ones. For example, Saifullina and her colleagues developed several heuristic algorithms to build and sample simple graphs from their second order degree sequences (\cite{Y1,Y2,Y3}).

As it turns out, the heuristic approach here is adequate, since as we will show, the graphicality problem for the second order degree sequences is {\em strongly {\em NP}-complete}.

\medskip\noindent Our main result is the following:

\begin{theorem}\label{th:NP}
The second order degree sequence problem in general is strongly {\em NP}-complete.
\end{theorem}
It is clear that the problem is a member of the class {\em NP}, since one can check in polynomial time whether a graph second order degree sequence is identical with the given double-sequence. In the next section we will demonstrate that it is {\em NP}-complete indeed; we will show some known {\em NP}-complete problems can be reduced to it.

\medskip\noindent The following is a similar problem: denote $D_2(v)$ the sum of the degrees of the neighbors of $v$ in the simple graph $G$. This value is always larger than $d_2(v)$ since, on one hand, $v$ itself occurs in $d(v)$ times, on the other hand the neighbors' neighbors may be overlapping. The problem to decide whether a pair $\bd, \bD_2$ is graphical will be called the {\bf \XY} problem. We will show in Theorem \ref{th:D2} that this problem is strongly {\em NP}-complete.

\section{The \BF\ problem} \label{sec:BF}

First, we are going to construct a new {\em NP}-complete problem, called  \BF, then we will show how to reduce this problem to an instance of second order degree sequence problem.

For that end, assume that we are given $n$ items with $w_1,w_2,\ldots, w_n$ positive integer weights such that the sum of the weights is $M.$ We also have $k$ baskets with capacity $(c_i, s_i)$ with the properties that
\begin{equation}\label{eq:basket}
\sum_{\ell=1}^k c_{\ell}= n \qquad \mbox{and}\qquad \sum_{\ell=1}^k s_{\ell} = M.
\end{equation}
\begin{definition}\label{def:BF}
The \BF\ problem with parameters $(w_1,\ldots, w_n)$ and $(c_1,s_1),\ldots (c_k,s_k)$ is to find an ordered partition $(C_1,\ldots, C_k)$ of the items, such that for $\ell=1, \ldots,k$ we have
\begin{displaymath}
\left \vert C_{\ell} \right \vert = c_{\ell} \quad\mbox{and}\quad \sum_{\mathrm{item}\in C_{\ell}} w(\mathrm{item}) = s_{\ell}.
\end{displaymath}
\end{definition}
\noindent It is clear that we can assume that each $c_\ell>1.$ (If this is not the case then we find a suitable filling for that basket, and forget the basket and its filler.)
\begin{theorem}\label{th:BF}
The {\rm \BF\ } problem is strongly {\em NP}-complete.
\end{theorem}
\begin{proof}
We show that the so called \3P\ problem can be reduced in polynomial time to the \BF\ problem. Let $W$ and $\alpha_1, \ldots \alpha_{3m}$ be positive integers with $W/4 < \alpha_i <W/2$ for each $i.$ Furthermore let $ \sum_i \alpha_i = mW.$ The  \3P\ problem is to decide whether the numbers $\alpha_i$ can be partitioned into $m$ classes  of integers can be partitioned into triples that all have the same sum $B.$ Due to the numerical conditions, in a successful partition all partition classes have exactly 3 elements. This problem is known to be strongly {\em NP}-complete. (Garey and Johnson \cite[Page 96]{GJ})

The polynomial reduction here is obvious: each basket has the parameter set $(3,W).$ A solution of this instance of the \BF\ problem is a solution of the \3P\ problem.
\end{proof}

\section{The proof of Theorem \ref{th:NP}}
\noindent We are going to give two slightly different reduction processes. The general instance of the \BF\ problem will be reduced to a general second order \DS\ problem, while the \3P\ problem will be reduced to a bipartite second order \DS\ problem.

To proceed with the proof of Theorem \ref{th:NP} assume we are given a \BF\ problem with the parameters described in Definition \ref{def:BF}. We are going to construct an instance of the second order degree sequence problem corresponding to our \BF\ problem.

In the graph there are four different types of vertices: $A_i^\ell, W_i, B_i, \Omega$.
\begin{enumerate}[{\rm (i)}]
\item For all $i=1,\ldots,n$ we have $w_i$ labeled {\bf atoms} $A_{\ell}^i$ $(\ell=1,\ldots,w_i)$. Altogether we have $\sum_i w_i$ atoms.
\item We have $n$ labeled {\bf weight} vertices $W_1,\ldots, W_n$.
\item We have $k$ labeled {\bf basket} vertices $B_1,\ldots, B_k.$
\item Finally we have one {\bf master} point $\Omega$.
\end{enumerate}
First we assume that
\begin{equation}\label{eq:cond1}
\forall i=1,\ldots,n; j=1,\ldots ,k \quad\mbox{we have}\quad w_i+1 \ne c_j;
\end{equation}
and
\begin{equation}\label{eq:cond2}
\forall i=1,\ldots,n; \quad\mbox{we have}\quad w_i+1 \ne n+k-1.
\end{equation}
The second order degree sequence $(\bd, \bd_2)$ is defined as follows:
\begin{enumerate}[{\rm (a)}]
\item For any fix $i=1,\ldots,n$ each atom $A_{\ell}^i$ has the second order degree sequence $(1,w_i+1).$
\item For each weight point $W_i$ its second order degree sequence is $(w_i+2, (k-1)+(n-1))$.
\item For each basket vertex $B_i$ its second order degree sequence is $(c_i+1, (k-1)+(n-c_i)+s_i)$.
\item Finally the master point $\Omega$ has second order degree sequence $(n+k, M).$
\end{enumerate}
Now assume that a graph $G$ is a realization of this particular BF-problem. By conditions (\ref{eq:cond1}) and (i) any atom $A_{\ell}^i$ must be adjacent to a weight point of degree $w_i+2.$ Since the weight points with the same degree cannot be distinguished therefore we just showed that $A_{\ell}^i$ is connected to $W_i.$ By condition (b) we have $d(W_i)=w_i+2$ therefore the $d_2(A_{\ell}^c) = (1, w_i+2 -1),$ as it required.

Since we determined the neighbors of all atoms,  the master point must be adjacent to all other points, so to all weight points and all basket vertices. Therefore $\Omega$ has $n+k$ neighbors, and its  second order degree is $\sum w_i = M$, as required.

We state that no two weight points are adjacent. Indeed, assume the opposite: if $W_i$ and $W_j$ are adjacent then $d_2(W_i)=(w_i+2, (n+k-2)+w_j).$ Since $w_j>0,$ it is a contradiction.

There are no two adjacent weight points, therefore each weight point is adjacent with exactly one basket vertex. Since each basket point is adjacent with the master vertex therefore by conditions (c) and (1) there are no two adjacent basket vertices. Finally that means that the adjacency relations between weight points and basket points in the graph $G$ provides a solution of our Basket Filling problem.

The remaining point is to handle the cases when conditions (\ref{eq:cond1}) or (\ref{eq:cond2}) do not hold. As already mentioned we can assume that all $c_i>1.$  Then we increase  the parameters $w_i$ and $s_i$ slightly (that is polynomially) as follows: we multiply all $w_i$ and $s_j$ with the value $(n+k) \max_{\ell} c_{\ell}.$
For this new instance conditions (\ref{eq:cond1}) or (\ref{eq:cond2}) hold automatically and any solution of the new problem provides a solution to the original problem as well. This finishes the proof of Theorem \ref{th:NP}. \hfill $\Box$
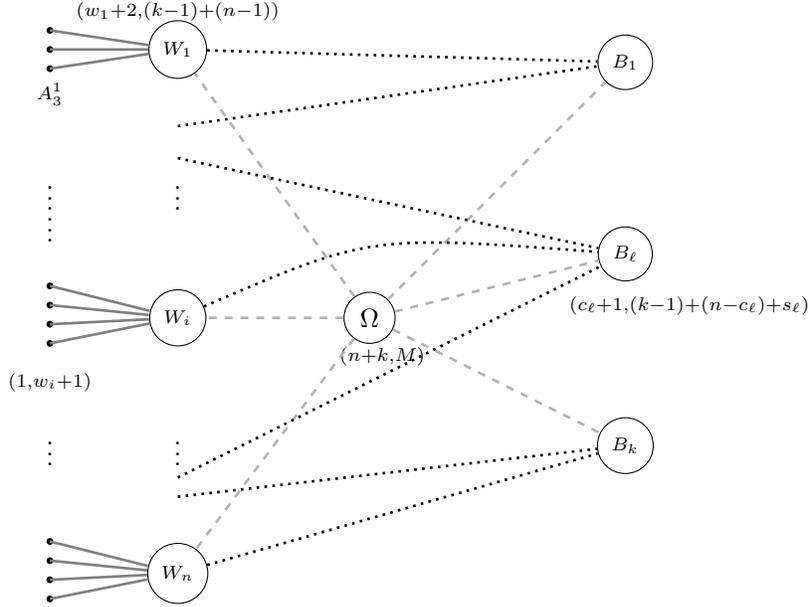
\begin{figure}
\begin{center}
\begin{tikzpicture}[scale=.85]
\fill [color=black] (0,8) circle (1.5pt);
\fill [color=black] (0,7.7) circle (1.5pt);
\fill [color=black] (0,7.4) circle (1.5pt);
\node at (0,7.0) {$\scriptstyle A_3^1$};
\node at (0,5.5) {$\vdots$};
\node at (0,5) {$\vdots$};
\fill [color=black] (0,4) circle (1.5pt);
\fill [color=black] (0,3.7) circle (1.5pt);
\fill [color=black] (0,3.4) circle (1.5pt);
\fill [color=black] (0,3.1) circle (1.5pt);
\node at (0,2.5) {$\scriptstyle (1,w_i+1 )$};
\node at (0,1.5) {$\vdots$};
\fill [color=black] (0,0) circle (1.5pt);
\fill [color=black] (0,-.3) circle (1.5pt);
\fill [color=black] (0,-.6) circle (1.5pt);
\fill [color=black] (0,-.9) circle (1.5pt);
\node (W_1) at (2,7.7)  [circle,draw] {$\scriptstyle W_1$} ;
\node at (2,8.3) {$\scriptstyle (w_1+2, (k-1)+(n-1))$};
\draw [line width=1pt,black!50] (0,8) -- (W_1);
\draw [line width=1pt,black!50] (0,7.7) -- (W_1);
\draw [line width=1pt,black!50] (0,7.4) -- (W_1);
\node at (2,5.5) {$\vdots$};
\node (W_i) at (2,3.5)  [circle,draw] {$\scriptstyle W_i$};
\draw [line width=1pt,black!50] (0,4) -- (W_i);
\draw [line width=1pt,black!50] (0,3.7) -- (W_i);
\draw [line width=1pt,black!50] (0,3.4) -- (W_i);
\draw [line width=1pt,black!50] (0,3.1) -- (W_i);
\node at (2,1.5) {$\vdots$};
\node (W_n) at (2,-.5)  [circle,draw] {$\scriptstyle W_n$};
\draw [line width=1pt,black!50] (0,0) -- (W_n);
\draw [line width=1pt,black!50] (0,-.3) -- (W_n);
\draw [line width=1pt,black!50] (0,-.6) -- (W_n);
\draw [line width=1pt,black!50] (0,-.9) -- (W_n);
\node (O) at (5,3.5) [circle,draw] {$\textstyle \Omega$};
\node at (5.2,2.9) {$\scriptstyle (n+k, M)$};
\draw [line width=1pt,dashed,black!30] (O) -- (W_1);
\draw [line width=1pt,dashed,black!30] (O) -- (W_i);
\draw [line width=1pt,dashed,black!30] (O) -- (W_n);
\node (B_1) at (9,7.5)  [circle,draw] {$\scriptstyle B_1$};
\node (B_l) at (9,4.5)  [circle,draw] {$\scriptstyle B_\ell$};
\node (B_k) at (9,1.5)  [circle,draw] {$\scriptstyle B_k$};
\draw [line width=1pt,dashed,black!30] (O) -- (B_1);
\draw [line width=1pt,dashed,black!30] (O) -- (B_l);
\draw [line width=1pt,dashed,black!30] (O) -- (B_k);
\draw [line width=1pt,dotted] (W_1) -- (B_1);
\draw [line width=1pt,dotted] (2,6.5) -- (B_1);
\draw [line width=1pt,dotted] (W_i) .. controls (5,4.8) .. (B_l);
\draw [line width=1pt,dotted] (2,6) -- (B_l);
\draw [line width=1pt,dotted] (2,1) -- (B_l);
\draw [line width=1pt,dotted] (W_n) -- (B_k);
\draw [line width=1pt,dotted] (2,.7) -- (B_k);
\node at (10,3.7) {$\scriptstyle (c_\ell+1, (k-1)+(n-c_\ell)+s_\ell)$};
\end{tikzpicture}
\caption{Second Order Degree representation of a \BF\ problem}
\end{center}
\end{figure}

\medskip\noindent
Now we present the second reduction process what shows that
\begin{theorem}\label{th:bi}
the second order degree sequence problem on bipartite graphs is also strongly {\em NP}-complete.
\end{theorem}
\begin{proof}
We will show that any \3P\ problem can be polynomially reduced to a second order degree sequence problem on bipartite graphs. So we are given a multiset $P$ of $n = 3 m$ positive integers $\alpha_i$ with $W/4 < \alpha_i < W/2,$ and we want to decide whether $P$ can be partitioned into $m$ triplets $P_1, P_2, \ldots, P_m$ such that the sum of the numbers in each subset is equal to $W$. In a successful solution each partition class consists of three numbers. We will assume that $W > 8.$

The construction of the bipartite second order degree sequence problem is very similar to the previous one. But we have a good use of the fact, that all baskets have the same weight capacity ($B = W$), and the same $c_{\ell}=3$ capacity, so we can get rid of the master point:

We will represent our partition problem with a bipartite graph what consists of $mW$ {\bf atoms}, $3m$ {\bf weight} points finally $m$ {\bf basket} points. In the graphical representation of a solution of the \3P\ problem each weight point will be connected to the necessary number of atoms and connected to exactly one basket point. Finally each basket point is adjacent with three weight points. By these conditions, the graph is automatically a bipartite one: one class contains all the atoms and basket points, while the other one consists of all the weight points.

The second order degree sequence is as follows:
\begin{enumerate}[{\rm (a)}]
\item Each atom belonging to weight point representing $\alpha_i$  has second order degree sequence $(1, \alpha_i).$ (The first neighbor is the weight point, the second neighbors are the other atoms, connected to the weight point, and one basket point.)
\item The weight point representing the number $\alpha_i$ has second order degree $(\alpha_i+1, 2).$ (The first neighbors are the atoms and one basket point, the second neighbors are the other two weight points connected to the basket point.)
\item A basket point has second order degree sequence $(3, W).$ (The neighbors are weight points, and the second neighbors are atoms.)
\end{enumerate}
Now assume that one can find a solution to the defined second order degree sequence problem. Then:
\begin{enumerate}[{\rm (1)}]
\item No atom can be connected to another atom (otherwise the second order degree is 0). No atom can be connected to a basket point (otherwise the second order degree is $2 < W/4$).
\item No weight point can be connected to another weight point (otherwise the second order degree is $\ge W/2-1 > 2$).
\item All basket points must be used against weight points.
\end{enumerate}
\end{proof}
\section{The relaxations of the Joint Degree Matrix problem}\label{sec:JDM}
In  recent years there has been a large (and growing) interest in real-life social and biological networks. One important distinction between these two network types lies in their overall structure: the first type typically have a few very high degree vertices and many low degree vertices with high {\em assortativity} (where a vertex is likely to be adjacent to vertices of similar degree), while the second kind is generally {\em disassortative} (in which low degree vertices tend to attach to those of high degree). It is well known, the {\em degree sequence} alone cannot capture these differences. There are several approaches to address this problem. One way to ease this problem is the JDM model (\cite{pinar,JDM,Mihail}):

Let $G=(V,E)$ be an $n$-vertex graph with {\em degree sequence} $\bd(G)=(d(v_1), \ldots,$ $d(v_n))$. We denote the maximum degree by $\Delta$, and for $1\le i\le\Delta$, the set of all vertices of degree $i$ is $V_i$. The {\em degree spectrum} $\bs_G(v)$ is a vector with $\Delta$ components, where $\bs_G(v)_i$ gives the number of vertices of degree $i$ adjacent to $v$ in the graph $G$. (This notion was originally introduced in \cite{JDM}.) While in {\em graphical realizations} of a degree sequence $\bd$ the degree of any particular vertex $v$ is prescribed, its degree spectrum may vary.

\begin{definition}\label{def}
The {\em joint degree matrix} {\rm (JDM)} $\J(G)=[\J_{ij}]$ of the graph $G$ is a $\Delta\times \Delta$ matrix where $\J_{ij}=\left |\{xy\in E(G):\,x\in V_i,y\in V_j\} \right |$. If, for a $k \times k$ matrix $M$ there exists a graph $G$ such that $\J(G)=M$, then $M$ is called a {\em graphical JDM}.
\end{definition}
There is an easy graphicality condition for the JDM model:
\begin{theorem}[Erd\H{o}s-Gallai type theorem for JDM, \cite{pinar}] \label{th:EG-JDM}
A $k \times k$ matrix $\J$ is a graphical JDM if and only if the followings hold:
\begin{enumerate}[{\rm (i)}]
\item for all $i:\; $ $n_i:=\frac{1}{i} \left (\J_{ii}+\sum\limits_{j=1}^{k} \J_{ij} \right )$ \ \ is an integer (this is actually $=|V_i|$);
\item for all $i:\; $ $\J_{ii}\le {n_i\choose 2}$;
\item for all $i\ne j:\; $ $\J_{ij}\le n_i n_j. \hfill \Box$
\end{enumerate}
\end{theorem}
\smallskip\noindent
Another way to measure the distance from "good" assortativity can be the following formulation (actually, similar parameters were suggested earlier). In the simple graph $G$ denote $D_2(v)$ the {\bf \XY}of vertex $v:$
\begin{equation}\label{eq:eq}
D_2(v) = \sum _{u \in \Gamma(v)} d(u).
\end{equation}
This is of course greater than $d_2(v)$ since $v$ is counted $d(v)$ times in it, furthermore there may be a lot of "overlapping" second neighbors. Now a graph should have high assortativity, if $D_2(v)$ is roughly $d(v)^2.$ Possible questions:

\smallskip
\begin{enumerate}[{\rm (A)}]
\item It is given $|V|$ and $D_2(v)$ but $\bd(G)$ is unknown. Is it graphical?
\item It is given $|V|$ and for all $i$ the values $\frac{\sum_{v\in V_i}D_2(v)}{i|V_i|}$ are known. Is it graphical?
\item The sequences $\bd(G), \bD_2(G)$ are given. Is it graphical?
\item The JDM matrix $\J$ and $\bD_2(G)$ are given. Is it graphical?
\item We know $\bd(G)$ and for each $i$ we know $\bD(i)=\sum_{v \in V_i} \bD_2(v).$ Is it graphical?
\end{enumerate}
\smallskip \noindent
Next we prove that problem (C) is {\em NP}-complete:
\begin{theorem}\label{th:D2}
The \XY problem defined by the pair $\bd(G),$ $\bD_2(G)$ is {\em NP}-complete.
\end{theorem}
\begin{proof}
We will describe an instance of the \3P\ problem with a similar graph gadget what was used in the proof of Theorem \ref{th:bi}: We have $mW$ atoms, $3m$ weight points and $m$ basket vertices, where each weight point has weight satisfying $W/4 < w < W/2$.

For the atoms: $d(A_i^\ell)=1$ and $D_2(A_i^\ell)=w_i+1.$ (There is only one neighbor, and it has degree $w_i+1.$) Then for the weight point $W_i$ we have: $d(W_i)=w_i+1$ (the extra degree is the basket vertex) and $D_2(W_i)=w_i+3.$ (Each atom next to $W_i$ has degree 1, and the neighboring basket vertex has three neighbors.) Finally for the basket vertex $B$ we have $d(B)=3$ and $D_2(B)=W+3$ since each weight point also incident with the basket vertex itself.

Assume that $G$ is a realization of this \XY problem. Then the atom $A_i^\ell$ must be adjacent to a weight point of weight $w_i.$ Then each weight point has just one free degree. Two weight points cannot be connected because then its $D_2(W_i)$ would be much greater then $w_i+3.$ So each basket vertex must be connected to three weight points.
\end{proof}

\smallskip\noindent We believe that problem (E) is also {\em NP}-complete. For that end we can set up the following integer feasibility region and we think that this special integer feasibility problem is {\em NP}-complete. More precisely:
\begin{problem}\label{imre}
Assume that we are given $|V_i|$ and $\bD(i)=\sum_{v \in V_i} \bD_2(v)$ for $i=1,\ldots, \Delta.$  We want to find a simple graph with these parameters.
\end{problem}
The problem is equivalent with finding a solution for the following integer feasibility problem.

For degree $i$ we have to find $ i |V_i|$ integers with $0\le j \le \Delta$  with sum $\bD(i)$ such that the defined Joint Degree Matrix is graphical. So we are looking for the unknown matrix $\J=[J_{i,j}]$ with dimensions $\Delta \times \Delta$ such that:
\begin{eqnarray}
&&\sum\limits_{j=1}^\Delta  j J_{i,j} + i J_{i,i,}  = \bD(i) \nonumber \\
&&\sum\limits_{j=1}^\Delta  J_{i,j} = i|V_i|  \nonumber \\
&& J_{i,j} = J_{j,i} \nonumber \\
&& 0  \le J_{i,j} \le |V_i| |V_j|\nonumber
\end{eqnarray}
The first two equalities describe the partition problem's constrain, the third one shows the symmetry of the JDM, while the last one ensures that the given JDM is graphical.

\end{document}